\documentclass[12pt,reqno,a4paper]{amsart}
\usepackage[english]{babel}
\usepackage[latin1]{inputenc}
\usepackage[T1]{fontenc}
\usepackage{palatino}
\usepackage{amsfonts}
\usepackage{amsmath}
\usepackage{amssymb}
\usepackage{amsthm}
\usepackage{graphicx}
\usepackage[colorlinks = true,citecolor = black]{hyperref}

\title[Weighted boundedness of the median maximal function]{A Characterization of the boundedness of the median maximal function on weighted $L^{p}$ spaces}
\author[Martikainen]{Henri Martikainen}
\address{Department of Mathematics and Statistics, University of Helsinki, Gustaf H\"allstr\"omin katu 2b, FI-00014 Helsinki, Finland}
\email{henri.martikainen@helsinki.fi} 
\thanks{Research of HM is supported by  the Academy of Finland grant 130166}

\author[Orponen] {Tuomas Orponen}
\address{Department of Mathematics and Statistics, University of Helsinki, Gustaf H\"allstr\"omin katu 2b, FI-00014 Helsinki, Finland}
\email{tuomas.orponen@helsinki.fi} 
\thanks{Research of TO is supported by the Finnish Centre of Excellence in Analysis and Dynamics Research}

\makeatletter
\@namedef{subjclassname@2010}{
  \textup{2010} Mathematics Subject Classification}
\makeatother

\subjclass[2010]{42B25 (Primary); 42B35 (Secondary)}
\keywords{Weighted norm inequalities, $A_p$ weights, maximal function, median}

\newcommand{\R}{\mathbb{R}}

\newcommand{\N}{\mathbb{N}}

\newcommand{\Z}{\mathbb{Z}}

\newcommand{\calD}{\mathcal{D}}

\newcommand{\calP}{\mathcal{P}}
\newcommand{\calM}{\mathcal{M}}

\theoremstyle{plain}
\newtheorem{thm}[equation]{Theorem}
\newtheorem{lemma}[equation]{Lemma}
\newtheorem{proposition}[equation]{Proposition}
\newtheorem{cor}[equation]{Corollary}

\theoremstyle{definition}

\theoremstyle{remark}
\newtheorem{remark}[equation]{Remark}

\numberwithin{equation}{section}

\begin{document}

\begin{abstract} We introduce and study the median maximal function $\calM f$, defined in the same manner as the classical Hardy-Littlewood maximal function, only replacing integral averages of $f$ by medians throughout the definition. This change has a qualitative impact on the mapping properties of the maximal operator: in contrast with the Hardy--Littlewood operator, which is not bounded on $L^{1}$, we prove that $\calM$ is bounded on $L^{p}(w)$ for all $0 < p < \infty$, if and only if $w \in A_{\infty}$. The characterization is purely qualitative and does not give the dependence on $[w]_{A_{\infty}}$.
However, the sharp bound $\|\calM\|_{L^{1}(w) \to L^{1}(w)} \lesssim [w]_{A_{1}}$ is established.
\end{abstract}

\maketitle

\section{Introduction}

Let $Q \subset \R^{n}$ be a cube. If $f \colon Q \to \R$ is a measurable function, any number $\alpha \in \R$ such that  
\begin{displaymath}
|Q \cap \{f < \alpha\}| \leq |Q|/2 \quad \text{and} \quad |Q \cap \{f > \alpha\}| \leq |Q|/2
\end{displaymath}
is called a \emph{median} of $f$ on $Q$. Here $|\cdot|$ is the Lebesgue measure.
One can check that the set of medians of $f$ on $Q$ forms a compact subinterval of $\R$. Hence the concept of the \emph{median with the largest absolute value},
denoted $m_{f}(Q)$, is well-defined. The notion of a median is a substitute for the average of the function on $Q$ and exists for any measurable function with no integrability assumptions (unlike the average).

As of late, medians have proved to be very useful in the weighted theory of singular integrals. This is mainly due to a formula discovered by Lerner \cite{Lerner1}.
For results and techniques related to this, see the papers by Lerner \cite{Lerner2} and Cruz--Uribe, Martell and P\'erez \cite{CMP}. However, our purpose in this note is to study
the \emph{median maximal function} $\calM$ on its own. This is defined by
\begin{displaymath}
\calM f(x) = \mathop{\sup_{Q \subset \R^n}}_{x \in Q} |m_{f}(Q)|,
\end{displaymath}
where $Q$ is a cube.

A weight $w$ is a non-negative measurable function. We recall the definitions of the Muckenhoupt weight classes $A_p$, $p \in [1, \infty]$.
One says that $w \in A_1$ if
\begin{displaymath}
[w]_{A_1} = \sup_Q \frac{w(Q)}{|Q|}\|w^{-1}\|_{L^{\infty}(Q)} < \infty.
\end{displaymath}
Here $w(Q) = \int_Q w$. Next, we have $w \in A_p$, $1 < p < \infty$, if
\begin{displaymath}
[w]_{A_p} = \sup_Q \frac{w(Q)}{|Q|} \Big(\frac{\sigma(Q)}{|Q|}\Big)^{p-1} < \infty,
\end{displaymath}
where $\sigma$ is the dual weight defined by $\sigma = w^{-1/(p-1)}$. Finally, $A_{\infty} = \bigcup_{p < \infty} A_p$. Alternatively, the class $A_{\infty}$ is characterized by the finiteness of either of the quantities
\begin{displaymath}
[w]_{A_{\infty}} = \sup_Q \frac{w(Q)}{|Q|} \exp \left(\frac{1}{|Q|} \int_{Q} \log w^{-1} \right)
\end{displaymath}
or
\begin{displaymath}
[w]_{A_{\infty}}' = \sup_Q \frac{1}{w(Q)} \int_Q M(w\chi_Q).
\end{displaymath}
Here $M$ is the Hardy--Littlewood maximal operator.
These distinctions are not too important for this note, as all our results regarding $A_{\infty}$ are purely qualitative. However, for possible further developments we mention that
$[w]_{A_{\infty}}' $ is the more attractive characteristic of the two, as $[w]_{A_{\infty}}'  \le C_n[w]_{A_{\infty}}$ and $[w]_{A_{\infty}}' $ can be exponentially smaller than $[w]_{A_{\infty}}$ (see the recent paper by
Hyt\"onen and P\'erez \cite{HP}).

Our main result is a characterization of the boundedness of $\calM$ in the weighted situation. Given a weight function $w \colon \R^{n} \to [0,\infty)$, we prove that $\calM$ is bounded on $L^{p}(w)$ for \emph{some} $0 < p < \infty$, if and only if it is bounded on $L^{p}(w)$ for \emph{all} $0 < p < \infty$, if and only if $w \in A_{\infty}$. This is in stark contrast with the mapping properties of the classical Hardy--Littlewood maximal operator $M$.
The characterization is qualitative in that the dependence on $[w]_{A_{\infty}}$ is not tracked (and the proof is such that it has no hope of giving sharp dependence). For results of similar nature, but involving maximal operators of different type, see
\cite{CN} and the references therein. On the quantitative side,
we establish the bound $\|\calM\|_{L^{1}(w) \to L^{1}(w)} \lesssim [w]_{A_{1}}$, which is best possible. We write $X \lesssim Y$ to mean
$X \le CY$ with some constant $C$. Sometimes we specify the dependence more carefully, that is, $X \lesssim_{\delta} Y$ means
$X \le C(\delta)Y$.

\subsection*{Acknowledgements}
We thank Tuomas Hyt\"onen, Pertti Mattila and Carlos P\'erez for comments and pointing to us some references.

\section{Auxiliary Maximal Functions} 

For technical purposes, we introduce the maximal operators $\calM^{\tau}$ for $0 < \tau < 1$. Given a cube $Q \subset \R^{n}$, define the number $m^{\tau}_{f}(Q) \in [0,\infty)$ by
\begin{displaymath} 
m^{\tau}_{f}(Q) = \sup \left\{\alpha \geq 0 : \frac{|Q \cap \{|f| \geq \alpha\}|}{|Q|} \geq \tau \right\}.
\end{displaymath}
A simple convergence argument shows that
\begin{displaymath}
\frac{|Q \cap \{|f| \geq m^{\tau}_{f}(Q)\}|}{|Q|} \geq \tau,
\end{displaymath} 
which means that the $\sup$ in the definition in $m_{f}^{\tau}(Q)$ is always attained. The maximal operator $\calM^{\tau}$ is now defined in the same manner as the median maximal operator:
\begin{displaymath}
\calM^{\tau}f(x) = \mathop{\sup_{Q \subset \R^n}}_{x \in Q} m^{\tau}_{f}(Q).
\end{displaymath}

The operators $\calM$ and $\calM^{1/2}$ are closely related:
\begin{proposition}\label{comparison} 
The inequality $|m_{f}(Q)| \leq m^{1/2}_{f}(Q)$ always holds, and so $\calM \leq \calM^{1/2}$.
\end{proposition}
\begin{proof} Let $\alpha \in \R$ be any median of $f$ on a cube $Q \in \calD$. If $\alpha \geq 0$, then
\begin{displaymath}
|Q \cap \{|f| \geq |\alpha|\}| \geq |Q \cap \{f \geq \alpha\}| \geq |Q|/2,
\end{displaymath}
and if $\alpha < 0$, similarly
\begin{displaymath}
|Q \cap \{|f| \geq |\alpha|\}| \geq |Q \cap \{f \leq \alpha\}| \geq |Q|/2.
\end{displaymath}
\noindent This shows that $m^{1/2}_{f}(Q) \geq |\alpha|$, and thus $m^{1/2}_{f}(Q) \geq |m_{f}(Q)|$.
\end{proof}
\begin{remark}
For non-negative functions $f \colon Q \to \R$ we have $m_{f}(Q) = m^{1/2}_{f}(Q)$.
\end{remark}

\subsection*{Dyadic variants}

Let $\calD$ be a dyadic grid in $\R^{n}$. That is, $\calD$ is a collection of cubes in $\R^{n}$ with the property that $Q_{1} \cap Q_{2} \in \{\emptyset,Q_{1},Q_{2}\}$ for all $Q_{1},Q_{1} \in \calD$. The $\calD$-dyadic variants of the maximal operators $\calM$ and $\calM^{\tau}$ are defined by
\begin{displaymath}
\calM_{\calD}f(x) = \mathop{\sup_{Q \in \calD}}_{x \in Q} |m_{f}(Q)| \quad \text{and} \quad \calM_{\calD}^{\tau}f(x) = \mathop{\sup_{Q \in \calD}}_{x \in Q} m^{\tau}_{f}(Q). 
\end{displaymath} 

As usual, one can control the original operators $\calM$ and $\calM^{\tau}$ by the maximum of the dyadic operators associated with several distinct grids $\calD$. Indeed, it is well-known that one may construct $2^{n}$ dyadic grids $\calD_{1},\ldots,\calD_{2^{n}}$ such that for \emph{any} (possibly non-dyadic) cube $Q \subset \R^{n}$ there exists a cube $R \in \bigcup_{j = 1}^{2^{n}} \calD_{j}$ with $Q \subset R$ and $|R| \leq C_{n}|Q|$, where $C_{n} \leq 6^{n}$. We record the following easy proposition:
\begin{proposition}\label{dyadic}
Let $\calD_{1},\ldots,\calD_{2^{n}}$ be the dyadic grids defined above. Then
\begin{displaymath}
\calM^{\tau}f(x) \leq \max \{ \calM^{C_{n}^{-1}\tau}_{\calD_{j}}f(x) : 1 \leq j \leq 2^{n}\}.
\end{displaymath} 
\end{proposition}

\section{Boundedness of $\calM$ and $\calM^{\tau}$ in weighted spaces}\label{twoWeight}

\subsection*{The $A_{\infty}$-Characterisation}

The following characterisation of $A_{\infty}$ is part (e) of \cite[Theorem 9.33]{Gr}:
\begin{lemma}\label{aInftyChar} A function $w \colon \R^{n} \to (0,\infty)$ is in $A_{\infty}$, if and only if there exist constants $\alpha,\beta \in (0,1)$ such that
\begin{displaymath}
|E| \geq \alpha|Q| \quad \Longrightarrow \quad w(E) \geq \beta w(Q)
\end{displaymath} 
for all cubes $Q \subset \R^{n}$ and all measurable subsets $E \subset Q$.
\end{lemma} 

\begin{thm} Let $\tau \in (0,1)$ and $p \in (0,\infty)$. Then the operator $\calM^{\tau}$ is bounded on $L^{p}(w)$, if and only if $w \in A_{\infty}$.
\end{thm} 

\begin{proof} Assume first that $\|\calM^{\tau}\|_{L^{p}(w) \to L^{p}(w)} < \infty$. If $Q \subset \R^{n}$ is any cube, and $E \subset Q$ is a subset with $|E| \geq \tau|Q|$, we have $(\calM^{\tau}\chi_{E})|Q \equiv 1$. Hence
\begin{displaymath}
w(Q)^{1/p} \leq \left(\int (\calM^{\tau}\chi_{E})^{p} \, dw\right)^{1/p} \leq \|\calM^{\tau}\|_{L^{p}(w) \to L^{p}(w)}w(E)^{1/p}.
\end{displaymath} 
This shows that the condition in Lemma \ref{aInftyChar} is in force for $w$ with $\alpha = \tau$ and $\beta = \|\calM\|_{L^{p}(w) \to L^{p}(w)}^{-p}$. Thus $w \in A_{\infty}$.

To prove the converse, let $w \in A_{\infty}$, and apply Lemma \ref{aInftyChar} to locate $\alpha,\beta \in (0,1)$ such that
\begin{equation}\label{form0}
|E| \geq \alpha|Q| \quad \Longrightarrow \quad w(E) \geq \beta w(Q)
\end{equation} 
for all cubes $Q \subset \R^{n}$ and all measurable subsets $E \subset Q$. Given $\eta \in (0,1)$ and a measurable subset $E \subset \R^{n}$, we will, for the rest of the proof, write $\calM^{\eta}(E)$ for the set satisfying $\calM^{\eta}(\chi_{E}) = \chi_{\calM^{\eta}(E)}$. More precisely,
\begin{displaymath}
\calM^{\eta}(E) = \{x \in \R^n:\, |E \cap Q| \ge \eta|Q|\textrm{ for some cube } Q \textrm{ containing } x\}.
\end{displaymath}

\begin{lemma}\label{cover} Let $\alpha \in (0,1)$ be the constant in (\ref{form0}). Then
\begin{displaymath} w(\calM^{\alpha}(E)) \lesssim_{w} w(E) \end{displaymath} 

\noindent for all bounded measurable sets $E \subset \R^{n}$. The set $\calM^{\alpha}(E)$ is also bounded and measurable.
\end{lemma}

\begin{proof} For every $x \in \calM^{\alpha}(E)$ we may find a cube $Q \subset \R^{n}$ such that $x \in Q$ and $|Q \cap E| \geq \alpha |Q|$. Use the basic covering lemma to find a disjoint collection $Q_{1},Q_{1},\ldots$ among these cubes such that $\calM^{\alpha}(E) \subset \bigcup_{j} 5Q_{j}$. The measure $w$ is doubling (see \cite[Proposition 9.3.2(6)]{Gr}), whence
\begin{displaymath}
w(\calM^{\alpha}(E)) \lesssim_{w} \sum_{j} w(Q_{j}) \leq \beta^{-1}\sum_{j} w(Q_{j} \cap E) \leq \beta^{-1}w(E),
\end{displaymath} 
using (\ref{form0}) in the penultimate inequality. The boundedness and measurability of $\calM^{\alpha}(E)$ are clear.
\end{proof}

\begin{lemma} Let $\eta \in (0,1)$, let $Q \subset \R^{n}$ be a cube, and let $E \subset Q$ be a measurable subset with $|E| \leq \eta |Q|$. Then 
\begin{displaymath}
|\calM^{\eta}(E) \cap Q| \geq \left(1 + \frac{\eta^{-1} - 1}{2^{n}}\right)|E|.
\end{displaymath}
\end{lemma}

\begin{proof} Let $\calD = \calD(Q)$, where $\calD(Q)$ is the collection of dyadic subcubes of $Q$, obtained by repeatedly dividing into $2^n$ equal cubes.
We may assume that $0 < |E| < \eta|Q|$. If $|E| = 0$, the claim is vacuous, and if $|E| = \eta|Q|$, then $Q \subset \calM^{\eta}(E)$, whence the claim follows. Let $\calP$ consist of the maximal cubes in $\calD$ such that $|E \cap P| \geq \eta |P|$ for every $P \in \calP$. By the Lebesgue differentiation theorem, such cubes exist and almost cover $E$, ie. $\sum_{P \in \calP} |P \cap E| = |E|$. Note that $Q \notin \calP$ by assumption. Let $\hat{P}$ denote the parent cube of $P$, and let $\hat{\calP} \subset \calD$ consist of the maximal cubes in the collection $\{\hat{P} : P \in \calP\}$. We will now demonstrate that
\begin{displaymath}
|\calM^{\eta}(E) \cap R| \geq \left(1 + \frac{\eta^{-1} - 1}{2^{n}}\right)|R \cap E|
\end{displaymath}
for every $R \in \hat{\calP}$. Summing over $R \in \hat{\calP}$ will then prove the lemma.

Fix $R \in \hat{\calP}$. Then $R = \hat{P}$ for some $P \in \calP$, which means that $|E \cap P| \geq \eta|P|$, but $|E \cap R| < \eta|R|$. Let $\lambda P$, $\lambda \in [1,2]$, denote the cube with the same common corner as $P$ and $R$ and side-length $\ell(\lambda P) = \lambda \ell(P)$. Then the function $\lambda \mapsto f(\lambda) := |E \cap \lambda P|/|\lambda P|$ is continuous, and $f(2) < \eta \leq f(1)$. Hence there exists $\lambda \in [1,2)$ such that $\tilde{P} = \lambda P$ satisfies $|E \cap \tilde{P}| = \eta|\tilde{P}|$. Then $\tilde{P} \subset M^{\eta}(E) \cap R$, whence 
\begin{align*}
|\calM^{\eta}(E) \cap R| & \geq |E \cap (R \setminus \tilde{P})| + |\tilde{P}|\\
& = |E \cap (R \setminus \tilde{P})| + \eta^{-1}|E \cap \tilde{P}|\\
& = |E \cap R| + (\eta^{-1} - 1)|E \cap \tilde{P}|\\
& \geq |E \cap R| + (\eta^{-1} - 1)\eta |P|\\
& = |E \cap R| + \frac{\eta^{-1} - 1}{2^{n}}\eta |R| \geq \left(1 + \frac{\eta^{-1} - 1}{2^{n}}\right)|E \cap R|.
\end{align*}
\end{proof}

\begin{lemma} Let $\gamma \in (0,1)$, and let $w$ be the $A_{\infty}$-weight satisfying (\ref{form0}). Then
\begin{displaymath}
|E| \geq \gamma|Q| \quad \Longrightarrow \quad w(E) \gtrsim_{\gamma,n,w} w(Q)
\end{displaymath} 
for all cubes $Q \subset \R^{n}$ and all measurable subsets $E \subset Q$. In other words, $\alpha$ can be essentially replaced by $\gamma$ in (\ref{form0}).
\end{lemma}

\begin{proof} We may assume that $\gamma \leq \alpha$. Fix a cube $Q \subset \R^{n}$ and a measurable subset $E \subset Q$ with $|E| \geq \gamma|Q|$. Let $(\calM^{\alpha})^{k}$ denote the $k^{th}$ iterate of the maximal operator $\calM^{\alpha}$. The previous lemma applied with $\eta = \alpha$ shows that either $|E| \geq \alpha|Q|$ (in which case everything is clear), or then
\begin{displaymath}
|(\calM^{\alpha})^{k}(E) \cap Q| \geq \alpha|Q|
\end{displaymath} 
for some large enough $k \in \N$ depending only on $\alpha,\gamma$ and $n$. Thus
\begin{displaymath}
w((\calM^{\alpha})^{k}(E) \cap Q) \geq \beta w(Q)
\end{displaymath} 
according to (\ref{form0}). Lemma \ref{cover} then yields
\begin{displaymath}
w(Q) \leq \beta^{-1}w((\calM^{\alpha})^{k}(E) \cap Q) \leq w((\calM^{\alpha})^{k}(E)) \lesssim_{\gamma,n,w} w(E),
\end{displaymath} 
as required.
\end{proof}

Now we are prepared to prove that $\calM^{\tau}$ is bounded on $L^{p}(w)$. Recall the discussion above Proposition \ref{dyadic}, and, in particular, the dyadic grids $\calD_{1},\ldots,\calD_{2^{n}}$ and the constant $C_{n}$. Applying the lemma above with $\gamma = C_{n}^{-1}\tau$ yields
\begin{equation}\label{comparison2}
|E| \geq C_{n}^{-1}\tau w(Q) \quad \Longrightarrow \quad w(E) \gtrsim_{n,w,\tau} w(Q)
\end{equation} 
for all cubes $Q \subset \R^{n}$ and all measurable subsets $E \subset Q$. Fix $1 \leq i \leq 2^{n}$ and write $\calD = \calD_{i}$. Given $f \in L^{p}(w)$ and $\lambda > 0$, the set $\{M_{\calD}^{C_{n}^{-1}\tau} f > \lambda\}$ is expressible as the disjoint union of the cubes $Q_{1},Q_{2},\ldots \in \calD$ maximal with respect to the property that $m_{f}^{C_{n}^{-1}\tau}(Q_{j}) > \lambda$. In particular,
\begin{displaymath}
|Q_{j} \cap \{|f| > \lambda\}| \geq |Q_{j} \cap \{|f| \geq m_{f}^{C_{n}^{-1}\tau}(Q_{j})\}| \geq C_{n}^{-1}\tau|Q_{j}|
\end{displaymath}
for $j \in \N$. Thus $E_{j} := Q_{j} \cap \{|f| > \lambda\}$ is subset of $Q_{j}$ to which (\ref{comparison2}) applies: 
\begin{displaymath}
w(\{M^{C_{n}^{-1}\tau}_{\calD}f > \lambda\}) = \sum_{j \in \N} w(Q_{j}) \lesssim_{n,w,\tau} \sum_{j \in \N} w(E_{j}) \leq w(\{|f| > \lambda\}).
\end{displaymath}
Thus
\begin{align*}
\|M_{\calD}^{C_{n}^{-1}\tau}f\|_{L^{p}(w)}^{p} & = \int_{0}^{\infty} p\lambda^{p - 1}w(\{M^{C_{n}^{-1}\tau}_{\calD}f > \lambda\}) \, d\lambda\\
& \lesssim_{n,w,\tau} \int_{0}^{\infty} p\lambda^{p - 1}w(\{|f| > \lambda\}) \, d\lambda = \|f\|_{L^{p}(w)}^{p},
\end{align*}
which proves that $\|M_{\calD}^{C_{n}^{-1}\tau}\|_{L^{p}(w) \to L^{p}(w)} < \infty$ for every dyadic grid $\calD = \calD_{i}$. Proposition \ref{dyadic} then implies that $\|M^{\tau}\|_{L^{p}(w) \to L^{p}(w)} < \infty$.
\end{proof}

We are ready to prove our main theorem.
\begin{thm} The median maximal operator $\calM$ is bounded on $L^{p}(w)$, $p \in (0,\infty)$, if and only if $w \in A_{\infty}$.
\end{thm}

\begin{proof} If $\calM$ is bounded on $L^{p}(w)$, the $A_{\infty}$-assertion then follows in the same manner as at the beginning of the proof of the previous theorem. Conversely, if $w \in A_{\infty}$, we may combine the previous theorem with the inequality $\calM \leq \calM^{1/2}$ to conclude that $\calM$ is bounded on $L^{p}(w)$.

\end{proof}
Fujii \cite{Fujii} proved that if $f \colon \R^n \to \R$ is a measurable function, then there holds $f(x) = \lim_{r \to 0} m_f(Q(x,r))$ for almost every $x$, where $Q(x,r)$ is the cube with center $x$ and side-length $2r$ (in fact, one may replace the cubes $Q(x,r)$ by any cubes converging to $x$). Combining this with the above theorem and dominated convergence yields the following corollary.
\begin{cor}
There holds
\begin{displaymath}
\lim_{r \to 0} \|f - m_f(Q(\cdot,r))\|_{L^p(w)} = 0
\end{displaymath}
for every $w \in A_{\infty}$ and $f \in L^p(w)$.
\end{cor}

\subsection*{Linear bound for the $A_1$-characteristic}
We prove below that the dependence on $[w]_{A_1}$ is linear and indicate that this is sharp. Let us consider only the dyadic case, as is clearly sufficient.
\begin{proposition}
Let $\calD$ be a dyadic grid. There holds
\begin{displaymath}
\|\calM^{\tau}_{\mathcal{D}}f\|_{L^1(w)} \le 4\tau^{-1}[w]_{A_1}\|f\|_{L^1(w)},
\end{displaymath}
and the linear dependence on $[w]_{A_1}$ is best possible.
\end{proposition}
\begin{proof}
For every $k \in \Z$ let $Q^k_1, Q^k_2, \ldots$ denote the maximal dyadic cubes for which $m_f^{\tau}(Q^k_j) > 2^k$. There holds
$|Q^k_j| \le \tau^{-1}|Q^k_j \cap \{|f| > 2^k\}|$. We may now estimate
\begin{align*}
\|\calM^{\tau}_{\mathcal{D}}f\|_{L^1(w)} &\le 2\sum_{k \in \Z} 2^kw(\{\calM^{\tau}_{\mathcal{D}}f > 2^k\}) \\
&= 2\sum_{k \in \Z} 2^k\sum_{j=1}^{\infty} w(Q^k_j) \\
&= 2\sum_{k \in \Z} 2^k\sum_{j=1}^{\infty} \frac{w(Q^k_j)}{|Q^k_j|}|Q^k_j| \\
&\le 2\sum_{k \in \Z} 2^k\sum_{j=1}^{\infty} [w]_{A_1}\inf_{Q^k_j} w \cdot \tau^{-1} |Q^k_j \cap \{|f| > 2^k\}| \\
&\le 2\tau^{-1}[w]_{A_1} \sum_{k \in \Z} 2^k\sum_{j=1}^{\infty} w(Q^k_j \cap \{|f| > 2^k\}) \\
&\le 2\tau^{-1}[w]_{A_1} \sum_{k \in \Z} 2^k w(\{|f| > 2^k\}) \\
&\le 4\tau^{-1}[w]_{A_1}\|f\|_{L^1(w)}.
\end{align*}

Let us now demonstrate the sharpness of this. Let $w_t = t\chi_{[-1,1]} + \chi_{\R \setminus [-1,1]}$, $t \in (0,1)$. One notes that $[w_t]_{A_1} = t^{-1}$.
We have
\begin{align*}
\|\calM^{1/2}\|_{L^1(w_t) \to L^1(w_t)} &\ge \frac{\|\calM^{1/2}\chi_{[-1,1]}\|_{L^1(w_t)}}{\|\chi_{[-1,1]}\|_{L^1(w_t)}} \\
&= \frac{\|\chi_{[-3,3]}\|_{L^1(w_t)}}{\|\chi_{[-1,1]}\|_{L^1(w_t)}} \\
&= \frac{4+2t}{2t} = \frac{2}{t} + 1 = 2[w_t]_{A_1} + 1.
\end{align*}
\end{proof}
\begin{remark}
The same proof yields the sharp bound $\|\calM_{\calD}^{\tau}\|_{L^{p}(w) \to L^{p}(w)} \lesssim_{\tau,p} [w]_{A_{1}}^{1/p}$.
\end{remark}

\end{document}